\documentclass[11pt,a4paper]{article}
\usepackage{indentfirst}
\usepackage{amsmath,amssymb,amsfonts,amsthm,graphics}
\usepackage{makeidx}
\usepackage{color}
\usepackage[dvips]{graphicx}
\usepackage{eufrak}
\usepackage{mathrsfs}

\newtheorem{theorem}{Theorem}

\newtheorem{proposition}{Proposition}
\newtheorem{lemma}{Lemma}

\newtheorem{corollary}{Corollary}

\newtheorem{observation}{Observation}
\newtheorem{conjecture}{Conjecture}
\begin{document}
\title{\Large \bf On the generalized (edge-)connectivity of
graphs\footnote{Supported by NSFC No.11071130, and ``the
Fundamental Research Funds for the Central Universities"}}
\author{\small Xueliang Li, Yaping Mao, Yuefang Sun
\\
\small Center for Combinatorics and LPMC-TJKLC
\\
\small Nankai University, Tianjin 300071, China
\\
\small lxl@nankai.edu.cn; maoyaping@ymail.com;
\small bruceseun@gmail.com}
\date{}
\maketitle
\begin{abstract}
The generalized $k$-connectivity $\kappa_k(G)$ of a graph $G$ was
introduced by Chartrand et al. in 1984. It is natural to introduce
the concept of generalized $k$-edge-connectivity $\lambda_k(G)$. For
general $k$, the generalized $k$-edge-connectivity of a complete
graph is obtained. For $k\geq 3$, tight upper and lower bounds of
$\kappa_k(G)$ and $\lambda_k(G)$ are given for a connected graph $G$
of order $n$, that is, $1\leq \kappa_k(G)\leq
n-\lceil\frac{k}{2}\rceil$ and $1\leq \lambda_k(G)\leq
n-\lceil\frac{k}{2}\rceil$. Graphs of order $n$ such that
$\kappa_k(G)=n-\lceil\frac{k}{2}\rceil$ and
$\lambda_k(G)=n-\lceil\frac{k}{2}\rceil$ are characterized,
respectively. Nordhaus-Gaddum-type results for the generalized
$k$-connectivity are also obtained. For $k=3$, we study the relation
between the edge-connectivity and the generalized
$3$-edge-connectivity of a graph. Upper and lower bounds of
$\lambda_3(G)$ for a graph $G$ in terms of the edge-connectivity
$\lambda$ of $G$ are obtained, that is, $\frac{3\lambda-2}{4}\leq
\lambda_3(G)\leq \lambda$, and two graph classes are given showing
that the upper and lower bounds are tight. From these bounds, we
obtain that $\lambda(G)-1\leq \lambda_3(G)\leq \lambda(G)$ if $G$ is
a connected planar graph, and the relation between the generalized
$3$-connectivity and generalized $3$-edge-connectivity of a graph
and its line graph.

{\flushleft\bf Keywords}: (edge-)connectivity, internally
(edge-)disjoint trees, generalized (edge-)connectivity, planar
graph, line graph, complementary graph.\\[2mm]
{\bf AMS subject classification 2010:} 05C40, 05C05, 05C75, 05C76.
\end{abstract}

\section{Introduction}

All graphs in this paper are undirected, finite and simple. We refer
to book \cite{bondy} for graph theoretical notation and terminology
not described here. The generalized connectivity of a graph $G$,
introduced by Chartrand et al. in \cite{Chartrand1}, is a natural
and nice generalization of the concept of (vertex-)connectivity. For
a graph $G=(V,E)$ and a set $S\subseteq V$ of at least two vertices,
\emph{an $S$-Steiner tree} or \emph{a Steiner tree connecting $S$}
(or simply, \emph{an $S$-tree}) is a such subgraph $T=(V',E')$ of
$G$ that is a tree with $S\subseteq V'$. Two Steiner trees $T$ and
$T'$ connecting $S$ are \emph{internally disjoint} if $E(T)\cap
E(T')=\varnothing$ and $V(T)\cap V(T')=S$. For $S\subseteq V(G)$,
the \emph{generalized local connectivity} $\kappa(S)$ of $S$ is the
maximum number of internally disjoint trees connecting $S$ in $G$.
The \emph{generalized $k$-connectivity} of $G$, denoted by
$\kappa_k(G)$, is then defined as $\kappa_k(G)=min\{\kappa(S) |
S\subseteq V(G) \  and \ |S|=k \}$. Thus, $\kappa_2(G)=\kappa(G)$.
Set $\kappa_k(G)=0$ when $G$ is disconnected. Results on the
generalized connectivity can be found in \cite{Chartrand4, LL, LLL,
LLL2, LLShi, LLSun, LLZ, Okamoto}.

A natural idea is to introduce the concept of generalized
edge-connectivity. Let $\lambda(S)$ denote the maximum number $\ell$
of pairwise edge-disjoint trees $T_1, T_2, \cdots,$ $ T_{\ell}$ in
$G$ such that $V(T_i)\supseteq S$ for every $1\leq i\leq \ell$ (Note
that these (edge-disjoint) trees are \emph{Steiner trees}). Then the
\emph{generalized $k$-edge-connectivity} $\lambda_k(G)$ of $G$ is
defined as $\lambda_k(G)= min\{\lambda(S) | S\subseteq V(G) \ and \
|S|=k\}$. Thus $\lambda_2(G)=\lambda(G)$. Set $\lambda_k(G)=0$ when
$G$ is disconnected.

The generalized edge-connectivity is related to an important
problem, which is called the \emph{Steiner Tree Packing Problem}.
For a given graph $G$ and $S\subseteq V(G)$, this problem asks to
find a set of maximum number of edge-disjoint Steiner trees
connecting $S$ in $G$. The difference between the Steiner Tree
Packing Problem and the generalized edge-connectivity is as follows:
The Steiner Tree Packing Problem studies local properties of graphs
since $S$ is given beforehand, but the generalized edge-connectivity
focuses on global properties of graphs since it first needs to find
the maximum number $\lambda(S)$ of edge-disjoint trees connecting
$S$ and then $S$ runs over all $k$-subsets of $V(G)$ to get the
minimum value of $\lambda(S)$.

The problem for $S=V(G)$ is called the \emph{Spanning Tree Packing
Problem} (Note that the Steiner Tree Packing Problem is a
generalization of the Spanning Tree Packing Problem). For any graph
$G$ of order $n$, the \emph{spanning tree packing number} or
\emph{$STP$ number}, is the maximum number of edge-disjoint spanning
trees contained in $G$. For the spanning tree packing number, Palmer
gave a good survey, see \cite{Palmer}. One can see that the $STP$
number of a graph $G$ is just $\kappa_n(G)$ or $\lambda_n(G)$.

In addition to being natural combinatorial measures, the generalized
connectivity and generalized edge-connectivity can be motivated by
their interesting interpretation in practice as well as theoretical
consideration.

From a theoretical perspective, both extremes of this problem are
fundamental theorems in combinatorics. One extreme of the problem is
when we have two terminals. In this case internally (edge-)disjoint
trees are just internally (edge-)disjoint paths between the two
terminals, and so the problem becomes the well-known Menger theorem.
The other extreme is when all the vertices are terminals. In this
case internally disjoint trees and edge-disjoint trees are just
spanning trees of the graph, and so the problem becomes the
classical Nash-Williams-Tutte theorem (for short proofs, see
\cite{Hakimi}).

\begin{theorem}(Nash-Williams \cite{Nash},Tutte \cite{Tutte})\label{th1}
A multigraph $G$ contains a system of $k$ edge-disjoint spanning
trees if and only if
$$
\|G/\mathscr{P}\|\geq k(|\mathscr{P}|-1)
$$
holds for every partition $\mathscr{P}$ of $V(G)$, where
$\|G/\mathscr{P}\|$ denotes the number of edges in $G$ between
distinct blocks of $\mathscr{P}$.
\end{theorem}

The following corollary is immediate from Theorem \ref{th1}.

\begin{corollary}\label{cor2}
Every $2\ell$-edge-connected graph contains a system of $\ell$
edge-disjoint spanning trees.
\end{corollary}

Kriesell \cite{Kriesell1} conjectured that this corollary can be
generalized for Steiner trees.

\begin{conjecture}(Kriesell \cite{Kriesell1})\label{con3}
If a set $S$ of vertices of $G$ is $2k$-edge-connected (See Section
$2$ for the definition), then there is a set of $k$ edge-disjoint
Steiner trees in $G$.
\end{conjecture}

Motivated by this conjecture, the Steiner Tree Packing Problem has
obtained wide attention and many results have been worked out, see
\cite{Jain, Kriesell1, Kriesell2, Lau, West}.

The generalized edge-connectivity and the Steiner Tree Packing
Problem have applications in $VLSI$ circuit design, see
\cite{Grotschel1, Grotschel2, Sherwani}. In this application, a
Steiner tree is needed to share an electronic signal by a set of
terminal nodes. Another application, which is our primary focus,
arises in the Internet Domain. Imagine that a given graph $G$
represents a network. We choose arbitrary $k$ vertices as nodes.
Suppose one of the nodes in $G$ is a \emph{broadcaster}, and all the
other nodes are either \emph{users} or \emph{routers} (also called
\emph{switches}). The broadcaster wants to broadcast as many streams
of movies as possible, so that the users have the maximum number of
choices. Each stream of movie is broadcasted via a tree connecting
all the users and the broadcaster. So, in essence we need to find
the maximum number Steiner trees connecting all the users and the
broadcaster, namely, we want to get $\lambda (S)$, where $S$ is the
set of the $k$ nodes. Clearly, it is a Steiner tree packing problem.
Furthermore, if we want to know whether for any $k$ nodes the
network $G$ has the above properties, then we need to compute
$\lambda_k(G)=min\{\lambda (S)\}$ in order to prescribe the
reliability and the security of the network.

For general $k$, the generalized $k$-edge-connectivity of a complete
graph is obtained. Tight upper and lower bounds of $\kappa_k(G)$ and
$\lambda_k(G)$ are given for a connected graph $G$ of order $n$,
that is, $1\leq \kappa_k(G)\leq n-\lceil\frac{k}{2}\rceil$ and
$1\leq \lambda_k(G)\leq n-\lceil\frac{k}{2}\rceil$.

By Nash-Williams-Tutte theorem, graphs of order $n$ such that
$\kappa_k(G)=n-\lceil\frac{k}{2}\rceil$ and
$\lambda_k(G)=n-\lceil\frac{k}{2}\rceil$ are characterized,
respectively. Nordhaus-Gaddum-type results for the generalized
$k$-connectivity are also obtained in Section $3$. For $k=3$, we
study the relation between the edge-connectivity and the generalized
$3$-edge-connectivity of a graph. Kriesell in \cite{Kriesell1}
showed that for any two natural numbers $t,\ell$ there exists a
smallest natural number $f_{\ell}(t)$ ($g_{\ell}(t)$) such that for
any $f_{\ell}(t)$-edge-connected ($g_{\ell}(t)$-edge-connected)
vertex set $S$ of a graph $G$ with $|S|\leq \ell$ ($|V(G)-S|\leq
\ell$) there exists a system $\mathscr{T}$ of $t$ edge-disjoint
trees such that $S\subseteq V(T)$ for each $T\in \mathscr{T}$. He
determined $f_3(t)=\lfloor\frac{8t+3}{6}\rfloor$. In Section 4, we
use his result to derive a tight lower bound of $\lambda_3(G)$. We
also give a tight upper bound of $\lambda_k(G)$. Altogether we get
that $\frac{3\lambda-2}{4}\leq \lambda_3(G)\leq \lambda$. Two graph
classes are given showing that the upper and lower bounds are tight.
From these bounds, we obtain two results: one is $\lambda(G)-1\leq
\lambda_3(G)\leq \lambda(G)$ if $G$ is a connected planar graph, the
other is the relation between the generalized $3$-connectivity and
generalized $3$-edge-connectivity of a graph and its line graph.

\section{Preliminaries}

For a graph $G$, let $V(G)$, $E(G)$, $|G|$, $\|G\|$, $L(G)$ and
$\overline{G}$ denote the set of vertices, the set of edges, the
order, the size, the line graph and the complement graph of $G$,
respectively. As usual, the {\it union} of two graphs $G$ and $H$ is
the graph, denoted by $G\cup H$, with vertex set $V(G)\cup V(H)$ and
edge set $E(G)\cup E(H)$. For $S\subseteq V(G)$, we denote by
$G\setminus S$ the subgraph obtained by deleting from $G$ the
vertices of $S$ together with the edges incident with them. If
$S=\{v\}$, we simply write $G\setminus v$ for $G\setminus \{v\}$. If
$S$ is a subset of vertices of a graph $G$, the subgraph of $G$
induced by $S$ is denoted by $G[S]$. If $M$ is the edge subset of
$G$, then $G\setminus M$ denotes the subgraph obtained by deleting
the edges of $M$ from $G$. $G\setminus \{e\}$ is abbreviated to
$G\setminus e$. If $M$ is a subset of edges of a graph $G$, the
subgraph of $G$ induced by $M$ is denoted by $G[M]$. We denote by
$E_G[X,Y]$ the set of edges of $G$ with one end in $X$ and the other
in $Y$. If $X=\{x\}$, we simply write $E_G[x,Y]$ for $E_G[\{x\},Y]$.

Chartrand et al. in \cite{Chartrand4} obtained the first result in
the generalized connectivity.

\begin{theorem}\cite{Chartrand4}\label{th2}
For every two integers $n$ and $k$ with $2\leq k\leq n$,
$$
\kappa_k(K_n)=n-\lceil k/2\rceil.
$$
\end{theorem}

For distinct vertices $x,y$ in $G$, let $\lambda(x,y;G)$ denote the
local edge-connectivity of $x$ and $y$. $S\subseteq V(G)$ is called
\emph{$n$-edge-connected}, if $\lambda(x,y;G)\geq n$ for all $x\neq
y$ in $S$. In \cite{Kriesell1}, Kriesell gave the following result.

\begin{lemma}\cite{Kriesell1}\label{lem1}
Let $t\geq 1$ be a natural number, and $G$ be a graph, and let
$\{a,b,c\}\subseteq V(G)$ be
$\lfloor\frac{8t+3}{6}\rfloor$-edge-connected in $G$. Then there
exists a system of $t$ edge-disjoint $\{a,b,c\}$-trees.
\end{lemma}

Chartrand \cite{Chartrand2} et al. investigated the relation between
the connectivity and edge-connectivity of a graph and its line
graph.

\begin{lemma}\cite{Chartrand2}\label{lem2}
If $G$ is a connected graph, then

$(1)$ $\kappa(L(G))\geq \lambda(G)$ if $\lambda(G)\geq 2$.

$(2)$ $\lambda(L(G))\geq 2\lambda(G)-2$.

$(3)$ $\kappa(L(L(G)))\geq 2\kappa(G)-2$.
\end{lemma}

Palmer \cite{Palmer} gave the $STP$ number of a complete bipartite
graph.

\begin{lemma}\cite{Palmer}\label{lem3}
The $STP$ number of a complete bipartite graph $K_{a,b}$ is
$\lfloor\frac{ab}{a+b-1}\rfloor$.
\end{lemma}

\section{Results of $\kappa_k(G)$ and $\lambda_k(G)$ for general $k$}

After the preparation of the above section, we start to give our
main results of this paper.

\subsection{Results for complete graphs}

The following two observations are easily seen.
\begin{observation}\label{obs1}
If $G$ be a connected graph, then $\kappa_k(G)\leq \lambda_k(G)\leq
\delta(G)$.
\end{observation}
\begin{observation}\label{obs2}
If $H$ is a spanning subgraph of $G$, then $\kappa_k(H)\leq
\kappa_k(G)$ and $\lambda_k(H)\leq \lambda_k(G)$.
\end{observation}

For a general $k$ and the complete graph $K_n$, $\kappa_k(K_n)$ was
determined by Chartrand et al.; see Theorem \ref{th2}. Now we give
the result for $\lambda_k(K_n)$.

Let $S=\{u_1,u_2,\cdots,u_k\}\subseteq V(G)$ and
$\bar{S}=\{w_1,w_2,\cdots,w_{n-k}\}$. Let $\mathscr{T}$ be a maximum
set of edge-disjoint trees in $G$ connecting $S$. Let
$\mathscr{T}_1$ be the set of trees in $\mathscr{T}$ whose edges
belong to $E(G[S])$, and $\mathscr{T}_2$ be the set of trees
containing at least one edge of $E_G[S,\bar{S}]$. Thus,
$\mathscr{T}=\mathscr{T}_1\cup \mathscr{T}_2$ (Throughout this
paper, $\mathscr{T}$, $\mathscr{T}_1$, $\mathscr{T}_2$ are always
defined as this).

\begin{lemma}\label{lem4}
Let $S\subseteq V(G)$, $|S|=k$ and $T$ be a tree connecting $S$. If
$T\in \mathscr{T}_1$, then $T$ uses $k-1$ edges of $E(G[S])\cup
E_G[S,\bar{S}]$. If $T\in \mathscr{T}_2$, then $T$ uses $k$ edges of
$E(G[S])\cup E_G[S,\bar{S}]$.
\end{lemma}

\begin{proof}
It is easy to see that for each tree $T$ in $\mathscr{T}_1$, $T$
uses $k-1$ edges in $E(G[S])$, namely, $T$ uses $k-1$ edges of
$E(G[S])\cup E_G[S,\bar{S}]$.

For $T\in \mathscr{T}_2$, by deleting all the vertices of $T$ from
$\bar{S}$, we obtain some components of $T$ in $S$, denoted by
$C_1,C_2,\cdots,C_s$. Let $|C_{i}|=c_{i}$. Then $|E(C_{i})|=c_{i}-1$
and $\sum_{i=1}^s(c_{i}-1)=k-s$. Since there exists one edge of $T$
between each $C_i$ and $\bar{S}$, where $1\leq i\leq s$, $T$ uses
$(k-s)+s=k$ edges in $E(G[S])\cup E_G[S,\bar{S}]$.
\end{proof}

\begin{theorem}\label{th3}
For every two integers $n$ and $k$ with $2\leq k\leq n$,
$$
\lambda_k(K_n)=n-\lceil k/2 \rceil.
$$
\end{theorem}
\begin{proof}
Let $G=K_n$. We choose $S\subseteq V(G)$ such that $|S|=k$. Let
$|\mathscr{T}|=y$ and $|\mathscr{T}_1|=x$. From Lemma \ref{lem4},
each tree $T\in \mathscr{T}_1$ uses $k-1$ edges in $E(G[S])\cup
E_G[S,\bar{S}]$, $|\mathscr{T}_1|=x\leq
\lfloor{{k}\choose{2}}/(k-1)\rfloor=\lfloor\frac{k}{2}\rfloor$.
Since each tree $T\in \mathscr{T}_2$ uses $k$ edges in $E(G[S])\cup
E_G[S,\bar{S}]$, we have $|\mathscr{T}_1|(k-1)+|\mathscr{T}_2|k\leq
|E_G[S,\bar{S}]|+|E(G[S])|$, that is, $x(k-1)+(y-x)k\leq
{{k}\choose{2}}+k(n-k)$. So $\lambda_k(G)\leq y\leq
\frac{k-1}{2}+n-k+\frac{x}{k}=n-\lceil\frac{k}{2}\rceil+\frac{x}{k}$
since $x\leq \lfloor\frac{k}{2}\rfloor$ and $y$ is an integer.

From the above arguments, we conclude that $\lambda_k(K_n)\leq
n-\lceil \frac{k}{2} \rceil$. Combining this with Theorem \ref{th2}
and Observation \ref{obs1}, we have $\lambda_k(K_n)=n-\lceil
\frac{k}{2} \rceil$.
\end{proof}

From Theorems \ref{th2} and \ref{th3}, we get that
$\lambda_k(G)=\kappa_k(G)$ for a complete graph $G=K_n$. However,
this is a very special case. Actually, $\lambda_k(G)-\kappa_k(G)$
could be very large. For example, let $G$ be a graph obtained from
two copies of the complete graph $K_n$ by identifying one vertex in
each of them. Then for $k<n$,
$\lambda_k(G)=n-\lceil\frac{k}{2}\rceil$, but $\kappa_k(G)=1$.

\subsection{Graphs with $\kappa_k(G)=n-\lceil k/2 \rceil$ and
$\lambda_k(G)=n-\lceil k/2 \rceil$, respectively}

At first, we give the tight bounds for $\kappa_k(G)$ and
$\lambda_k(G)$:

\begin{proposition}\label{pro1}
For a connected graph $G$ of order $n$ and $3\leq k\leq n$, $1\leq
\kappa_k(G)\leq n-\lceil k/2 \rceil$. Moreover, the upper and lower
bounds are tight.
\end{proposition}
\begin{proof}
From Observation \ref{obs2} and Theorem \ref{th2}, we have
$\kappa_k(G)\leq \kappa_k(K_n)=n-\lceil\frac{k}{2}\rceil$. Since $G$
is connected, then $\kappa_k(G)\geq 1$. The result holds.

One can easily check that the complete graph $K_n$ attains the upper
bound and any tree $T_{n}$ on $n$ vertices attains the lower bound.
\end{proof}

The same upper and lower bounds can be established for the
generalized $k$-edge-connectivity.

\begin{proposition}\label{pro2}
For a connected graph $G$ of order $n$ and $3\leq k\leq n$, $1\leq
\lambda_k(G)\leq n-\lceil k/2 \rceil$. Moreover, the upper and lower
bounds are tight.
\end{proposition}

Next, we will characterize graphs with $\kappa_k(G)=
n-\lceil\frac{k}{2}\rceil$ and $\lambda_k(G)=
n-\lceil\frac{k}{2}\rceil$, respectively. Let us start with some
lemmas, which will be used later.

\begin{lemma}\label{lem5}
For an even $k$ with $4\leq k\leq n$, $\lambda_k(K_n\setminus e)<
n-\frac{k}{2}$ for any $e\in E(K_n)$.
\end{lemma}

\begin{proof}
Let $G=K_n\setminus e$. We choose $S\subseteq V(G)$ such that
$|S|=k$ and $K_n[S]$ containing $e$. Let $|\mathscr{T}|=y$ and
$|\mathscr{T}_1|=x$. Since every tree $T\in \mathscr{T}_1$ uses
$k-1$ edges in $E(G[S])\cup E_G[S,\bar{S}]$, $|\mathscr{T}_1|=x\leq
\big({{k}\choose{2}}-1\big)/(k-1)= \frac{k}{2}-\frac{1}{k-1}$. From
Lemma \ref{lem4}, each tree $T\in \mathscr{T}_2$ uses $k$ edges of
$E(G[S])\cup E_G[S,\bar{S}]$. Thus
$|\mathscr{T}_1|(k-1)+|\mathscr{T}_2|k\leq
|E_G[S,\bar{S}]|+|E(G[S])|$, that is, $x(k-1)+(y-x)k\leq
{{k}\choose{2}}+k(n-k)-1$. So $\lambda_k(G)=y\leq
\frac{k-1}{2}+n-k+\frac{x-1}{k}\leq
n-\frac{k}{2}-\frac{1}{k-1}<n-\frac{k}{2}$.
\end{proof}

\begin{lemma}\label{lem6}
If $k$ is odd with $3\leq k\leq n$, and $M$ is an edge set of the
complete graph $K_n$ such that $|M|\geq \frac{k+1}{2}$, then
$\lambda_k(K_n\setminus M)< n-\frac{k+1}{2}$.
\end{lemma}

\begin{proof}
Let $G=K_n\setminus M$. We can choose $S\subseteq V(G)$ such that
$|S|=k$ and $|M\cap \big(E(K_n[S])\cup E_{K_n}[S,\bar{S}])|\geq
\frac{k+1}{2}$. Let $|\mathscr{T}|=y$ and $|\mathscr{T}_1|=x$. Since
each tree $T\in \mathscr{T}_1$ uses $k-1$ edges in $E(G[S])\cup
E_G[S,\bar{S}]$, $|\mathscr{T}_1|=x\leq {{k}\choose{2}}/(k-1)=
\frac{k-1}{2}$. From Lemma \ref{lem4}, each tree $T\in
\mathscr{T}_2$ uses $k$ edges of $E(G[S])\cup E_G[S,\bar{S}]$. Thus
$|\mathscr{T}_1|(k-1)+|\mathscr{T}_2|k\leq
|E_G[S,\bar{S}]|+|E(G[S])|$, that is, $x(k-1)+(y-x)k\leq
{{k}\choose{2}}+k(n-k)-\frac{k+1}{2}$. So $\lambda_k(G)=y\leq
\frac{k-1}{2}+n-k+\frac{x}{k}-\frac{k+1}{2k}\leq
n-\frac{k+1}{2}-\frac{1}{2k}<n-\frac{k+1}{2}$.
\end{proof}

\begin{lemma}\label{lem7}
If $n$ is odd and $M$ is an edge set of the complete graph $K_n$
such that $0\leq |M|\leq \frac{n-1}{2}$, then $G=K_n\setminus M$
contains $\frac{n-1}{2}$ edge-disjoint spanning trees.
\end{lemma}
\begin{proof}
Let $\mathscr{P}=\bigcup_{i=1}^pV_i$ be a partition of $V(G)$ with
$|V_i|=n_i \ (1\leq i\leq p)$, and $\mathcal {E}_p$ be the set of
edges between distinct blocks of $\mathscr{P}$ in $G$. The case
$p=1$ is trivial, thus we assume $p\geq 2$. Then $|\mathcal
{E}_p|\geq {{n}\choose{2}}-\sum_{i=1}^p{{n_i}\choose{2}}-|M|\geq
{{n}\choose{2}}-\sum_{i=1}^p{{n_i}\choose{2}}-\frac{n-1}{2}$. We
will show that
${{n}\choose{2}}-\sum_{i=1}^p{{n_i}\choose{2}}-\frac{n-1}{2}\geq
\frac{n-1}{2}(p-1)$, that is, $(n-p)\frac{n-1}{2}\geq
\sum_{i=1}^p{{n_i}\choose{2}}$. We only need to prove that
$(n-p)\frac{n-1}{2}\geq max\{\sum_{i=1}^p{{n_i}\choose{2}}\}$. Since
$f(n_1,n_2,\cdots,n_p)=\sum_{i=1}^p{{n_i}\choose{2}}$ obtains its
maximum value when $n_1=n_2=\cdots=n_{p-1}=1$ and $n_p=n-p+1$, we
need to show the inequality $(n-p)\frac{n-1}{2}\geq
{{1}\choose{2}}(p-1)+{{n-p+1}\choose{2}}$, that is
$(n-p)\frac{p-2}{2}\geq 0$. It is easy to see that the inequality
holds. Thus, $|\mathcal {E}_p|\geq
{{n}\choose{2}}-\sum_{i=1}^p{{n_i}\choose{2}}-|M|\geq
\frac{n-1}{2}(p-1)$. From Theorem $1$, we know that there exist
$\frac{n-1}{2}$ edge-disjoint spanning trees (Note that we can use
the result of Theorem \ref{th1}, although Nash-Williams and Tutte
considered multigraphs but here we are concerned with the
generalized connectivity and generalized edge-connectivity for
simple graphs).
\end{proof}

\begin{theorem}\label{th4}
Let $G$ be a connected graph of order $n$ and $k$ be an integer such
that $3\leq k\leq n$. Then $\kappa_k(G)=n-\lceil\frac{k}{2}\rceil$
if and only if $G=K_n$ for $k$ even; $G=K_n\setminus M$ for $k$ odd,
where $M$ is an edge set such that $0\leq |M|\leq \frac{k-1}{2}$.
\end{theorem}

\begin{proof}
First we consider the case that $k$ is even. From Theorem \ref{th2},
we have $\kappa_k(K_n)=n-\frac{k}{2}$. Actually, the complete graph
$K_n$ is the unique graph with this property. We only need to show
that $\kappa_k(K_n\setminus e)<n-\frac{k}{2}$ for any $e\in E(K_n)$.
From Lemma \ref{lem5} and Observation \ref{obs1}, we know that
$\kappa_k(K_n\setminus e)\leq \lambda_k(K_n\setminus e)<
n-\frac{k}{2}$ for $e\in E(K_n)$. Thus, the result holds for $k$
even.

Next we consider the case that $k$ is odd.

\emph{Necessity:} Let $G$ be a graph of order $n$ such that
$\kappa_k(G)=n-\frac{k+1}{2}$. Since $G$ is connected, we can
consider $G$ as a graph obtained by deleting some edges from the
complete graph $K_n$. If $G=K_n\setminus M$ such that $|M|\geq
\frac{k+1}{2}$, then $\kappa_k(K_n\setminus M)\leq
\lambda_k(K_n\setminus M)<n-\frac{k+1}{2}$ by Observation \ref{obs1}
and Lemma \ref{lem6}, a contradiction. Thus, $G=K_n\setminus M$,
where $0\leq |M|\leq \frac{k-1}{2}$.

\emph{Sufficiency:}~~ We will show that $\kappa_k(G)\geq
n-\frac{k+1}{2}$ if $G=K_n\setminus M$ such that $0\leq |M|\leq
\frac{k-1}{2}$. It suffices to prove that $\kappa_k(G)\geq
n-\frac{k+1}{2}$ for $|M|=\frac{k-1}{2}$.

Let $S=\{u_1,u_2,\cdots,u_k\}\subseteq V(G)$ and
$\bar{S}=\{w_1,w_2,\cdots,w_{n-k}\}$. We have the following two
cases to consider:

\textbf{Case 1}. $M\subseteq E(K_n[S])\cup E(K_n[\bar{S}])$.

Let $M'=M\cap E(K_n[S])$ and $M''=M\cap E(K_n[\bar{S}])$. Then
$|M'|+|M''|=|M|=\frac{k-1}{2}$ and $0\leq |M'|,|M''|\leq
\frac{k-1}{2}$. We can consider $G[S]$ as a graph obtained by
deleting $|M'|$ edges from the complete graph $K_k$. From Lemma
\ref{lem7}, there exist $\frac{k-1}{2}$ edge-disjoint spanning trees
in $G[S]$. Actually, these $\frac{k-1}{2}$ edge-disjoint trees are
all trees connecting $S$ in $G[S]$. All these trees together with
the trees $T_i=w_iu_1\cup w_iu_2\cup \cdots \cup w_iu_{k} \ (1\leq
i\leq n-k)$ form $n-\frac{k+1}{2}$ internally disjoint trees
connecting $S$, namely, $\kappa(S)\geq n-\frac{k+1}{2}$ (Note that
the trees connecting $S$ can be edge-disjoint in $G[S]$, but must be
internally disjoint in $G\setminus S$).

\textbf{Case 2}. $M\nsubseteq E(K_n[S])\cup E(K_n[\bar{S}])$.

In this case, there exist some edges of $M$ in $E_{K_n}[S,\bar{S}]$.
Let $M'=M\cap E(K_n[S])$ and $M''=M\cap E(K_n[\bar{S}])$, and let
$|M'|=m_1$ and $|M''|=m_2$. Clearly, $0\leq m_i\leq \frac{k-3}{2} \
(i=1,2)$.

For $w_i\in \bar{S}$, we let $|E_{K_n[M]}[w_i,S]|=x_i$, where $1\leq
i\leq n-k$. Without loss of generality, let $x_1\geq x_2\geq \cdots
\geq x_{n-k}$. Thus $\sum_{i=1}^{n-k}x_i+m_1+m_2=\frac{k-1}{2}$ and
$|E_{G}[w_i,S]|=k-x_i$.

Our basic idea is to seek for some edges in $G[S]$, and let them
together with the edges of $E_G[S,\bar{S}]$ form $n-k$ internally
disjoint trees connecting $S$.

For $w_1\in \bar{S}$, without loss of generality, let
$S_1=\{u_1,u_2,\cdots,u_{x_1}\}$ such that $u_jw_1\in M \ (1\leq
j\leq x_1)$ and $S_2=S\setminus
S_1=\{u_{x_1+1},u_{x_1+2},\cdots,u_{k}\}$. Clearly, $S=S_1\cup S_2$
and $u_jw_1\in E(G) \ (x_1+1\leq j\leq k)$, namely,
$S_2=N_G(w_1)\cap S$. One can see that the tree
$T_1'=w_1u_{x_1+1}\cup w_1u_{x_1+2}\cup \cdots \cup w_1u_{k}$ is a
Steiner tree connecting $S_2$. Our idea is to seek for $x_1$ edges
in $E_G[S_1,S_2]$ and add them to $T_1'$ to form a Steiner tree
connecting $S$. For each $u_j\in S_1 \ (1\leq j\leq x_1)$, we claim
that $|E_G[u_j,S_2]|\geq 1$. Otherwise, let $|E_G[u_j,S_2]|=0$. Then
$|E_{K_n[M]}[u_j,S_2]|=k-x_1$ and $|M|\geq
|E_{K_n[M]}[u_j,S_2]|+d_{K_n[M]}(w_1)\geq (k-x_1)+x_1=k$, which
contradicts to $|M|=\frac{k-1}{2}$. Since $|E_G[u_j,S_2]|\geq 1$ for
each $u_j \ (1\leq j\leq x_1)$, we can find a vertex $u_r \
(x_1+1\leq r\leq k)$ such that $e_{1j}=u_ju_r\in E(G[S])$. Let
$M_1=\{e_{11},e_{12},\cdots,e_{1x_1}\}$ and $G_1=G\setminus M_1$.
Thus the tree $T_1=w_1u_{x_1+1}\cup w_1u_{x_1+2}\cup \cdots \cup
w_1u_{k}\cup e_{11}\cup e_{12}\cup \cdots \cup e_{1x_1}$ is our
desired one.

For $w_2\in \bar{S}$, without loss of generality, let
$S_1=\{u_1,u_2,\cdots,u_{x_2}\}$ such that $u_jw_2\in M \ (1\leq
j\leq x_2)$ and $S_2=S\setminus
S_1=\{u_{x_2+1},u_{x_2+2},\cdots,u_{k}\}$. Clearly, $S=S_1\cup S_2$
and $u_jw_2\in E(G) \ (x_2+1\leq j\leq k)$, namely,
$S_2=N_G(w_2)\cap S$. One can see that the tree
$T_2'=w_2u_{x_2+1}\cup w_2u_{x_2+2}\cup \cdots \cup w_2u_{k}$ is a
Steiner tree connecting $S_2$. Our idea is to seek for $x_2$ edges
in $E_{G_1}[S_1,S_2]$ and add them to $T_2'$ to form a Steiner tree
connecting $S$. For each $u_j\in S_1 \ (1\leq j\leq x_2)$, we claim
that $|E_{G_1}[u_j,S_2]|\geq 1$. Otherwise, we let
$|E_{G_1}[u_j,S_2]|=0$. For $e\notin E_{G_1}[u_j,S_2]$, $e\in M$ or
$e\in M_1=\{e_{11},e_{12},\cdots,e_{1x_1}\}$. Then
$|E_{K_n[M]}[u_j,S_2]|\geq k-x_2-x_1$ and $|M|\geq
|E_{K_n[M]}[u_j,S_2]|+d_{K_n[M]}(w_1)+d_{K_n[M]}(w_2) \geq
(k-x_2-x_1)+x_1+x_2=k$, which contradicts to $|M|=\frac{k-1}{2}$.
Since $|E_{G_1}[u_j,S_2]|\geq 1$ for each $u_j \ (1\leq j\leq x_2)$,
we can find a vertex $u_r \ (x_2+1\leq r\leq k)$ such that
$e_{2j}=u_ju_r\in E(G_1[S])$. Let
$M_2=\{e_{21},e_{22},\cdots,e_{2x_2}\}$ and $G_2=G_1\setminus M_2$.
Thus the tree $T_2=w_2u_{x_2+1}\cup w_2u_{x_2+2}\cup \cdots \cup
w_2u_{k}\cup e_{21}\cup e_{22}\cup \cdots \cup e_{2x_2}$ is our
desired tree. Clearly, $T_2$ and $T_1$ are two edge-disjoint trees
connecting $S$.

For $w_i\in \bar{S} \ (3\leq i\leq n-k)$, without loss of
generality, let $S_1=\{u_1,u_2,\cdots,u_{x_i}\}$ such that
$u_jw_i\in M \ (1\leq j\leq x_i)$ and $S_2=S\setminus
S_1=\{u_{x_i+1},u_{x_i+2},\cdots,u_{k}\}$. Clearly, $S=S_1\cup S_2$
and $w_iu_j\in E(G) \ (x_i+1\leq j\leq k)$, namely,
$S_2=N_G(w_i)\cap S$. One can see the tree $T_i'=w_iu_{x_i+1}\cup
w_iu_{x_i+2}\cup \cdots \cup w_iu_{k}$ is a Steiner tree connecting
$S_2$. Our idea is to seek for $x_i$ edges in $E_{G_{i-1}}[S_1,S_2]$
and add them to $T_i'$ to form a Steiner tree connecting $S$. For
each $u_j\in S_1 \ (1\leq j\leq x_i)$, we claim that
$|E_{G_{i-1}}[u_j,S_2]|\geq 1$. Otherwise, let
$|E_{G_{i-1}}[u_j,S_2]|=0$. For $e\notin E_{G_{i-1}}[u_j,S_2]$, we
have that $e\in M$ or $e\in \bigcup_{r=1}^{i-1}M_r$. Then
$|E_{K_n[M]}[u_j,S_2]|\geq k-x_i-\sum_r^{i-1}x_r$ and $|M|\geq
|E_{K_n[M]}[u_j,S_2]|+\sum_r^id_{K_n[M]}(w_r) \geq
(k-\sum_r^ix_r)+\sum_r^ix_r=k$, which contradicts to
$|M|=\frac{k-1}{2}$. Since $|E_{G_{i-1}}[u_j,S_2]|\geq 1$ for each
$u_j \ (1\leq j\leq x_i)$, we can find a vertex $u_r \ (x_i+1\leq
r\leq k)$ such that $e_{ij}=u_ju_r\in E(G_{i-1}[S])$. Let
$M_i=\{e_{i1},e_{i2},\cdots,e_{ix_i}\}$ and $G_i=G_{i-1}\setminus
M_{i}$. Thus the tree $T_i=w_iu_{x_i+1}\cup w_iu_{x_i+2}\cup \cdots
\cup w_iu_{k}\cup e_{i1}\cup e_{i2}\cup \cdots \cup e_{ix_i}$ is our
desired one (Note that if $x_{i}=0$ then we do not need to search
for some edges of $E(G_{i-1}[S])$ and $T_i=w_iu_{1}\cup w_iu_{2}\cup
\cdots \cup w_iu_{k}$ is our desired tree). Clearly, $T_i$ and $T_j
\ (1\leq j\leq i-1)$ are two edge-disjoint trees connecting $S$.

We continue this procedure until we find out $n-k$ trees connecting
$S$, say $T_1,T_2,\cdots, T_{n-k}$. Now we terminate this procedure.
Clearly, we can consider $G_{n-k}[S]=G[S]\setminus
\bigcup_{i=1}^{n-k}M_i$ as a graph obtained by deleting
$|M'|+\sum_{i=1}^{n-k}|M_i|$ edges from the complete graph $K_k$.
Since $\sum_{i=1}^{n-k}x_i+m_1+m_2=\frac{k-1}{2}$, we have $1\leq
\sum_{i=1}^{n-k}|M_i|+m_1\leq \frac{k-1}{2}$. From Lemma \ref{lem7},
there exist $\frac{k-1}{2}$ edge-disjoint trees connecting $S$ in
$G[S]$ (Note that these trees can be edge-disjoint by the definition
of generalized $k$-connectivity). These trees together with
$T_1,T_2,\cdots,T_{n-k}$ form $n-\frac{k+1}{2}$ internally disjoint
trees connecting $S$, namely, $\kappa(S)\geq n-\frac{k+1}{2}$.

From the above discussion, we get that $\kappa(S)\geq
n-\frac{k+1}{2}$ for $S\subseteq V(G)$, which implies that
$\kappa_k(G)\geq n-\frac{k+1}{2}$. From this together with
Proposition \ref{pro1}, we have $\kappa_k(G)= n-\frac{k+1}{2}$.
\end{proof}

\begin{theorem}\label{th5}
For a connected graph $G$ of order $n$ and $n\geq k\geq 3$,
$\lambda_k(G)=n-\lceil\frac{k}{2}\rceil$ if and only if $G=K_n$ for
$k$ even; $G=K_n\setminus M$ for $k$ odd, where $M$ is an edge set
such that $0\leq |M|\leq \frac{k-1}{2}$.
\end{theorem}
\begin{proof}
First we consider the case that $k$ is even. From Proposition
\ref{pro2} and Lemma \ref{lem5}, we have that
$\lambda_k(K_n)=n-\frac{k}{2}$ if and only if $G=K_n$.

Next we consider the case that $k$ is odd. If $G=K_n\setminus M \
(0\leq |M|\leq \frac{k-1}{2})$, then $\lambda_k(G)\geq
\kappa_k(G)=n-\frac{k+1}{2}$ by Observation \ref{obs1} and Theorem
\ref{th4}. From this together with Proposition \ref{pro2}, we know
that $\lambda_k(G)=n-\frac{k+1}{2}$. Conversely, assume that
$\lambda_k(G)=n-\frac{k+1}{2}$. Since $G$ is connected, we can
consider $G$ as a graph obtained by deleting some edges from the
complete graph $K_n$. If $G=K_n\setminus M$ such that $|M|\geq
\frac{k+1}{2}$, then $\lambda_k(G)<n-\frac{k+1}{2}$ by Lemma
\ref{lem6}, a contradiction. So $G=K_n\setminus M$, where $0\leq
|M|\leq \frac{k-1}{2}$.
\end{proof}

\noindent \textbf{Remark 1}. The graphs with
$\kappa_k(G)=n-\lceil\frac{k}{2}\rceil$ or
$\lambda_k(G)=n-\lceil\frac{k}{2}\rceil$ has been characterized by
Theorems \ref{th4} and \ref{th5}. A natural question is, for the
lower bounds, whether we can characterize the graphs with
$\kappa_k(G)=1$ or $\lambda_k(G)=1$. It seems not easy to solve such
a problem. Note that the minimal graphs with $\kappa_k(G)=1$ or
$\lambda_k(G)=1$ are the trees of order $n$. So, an interesting
problem could be what is the maximal graphs with $\kappa_k(G)=1$ or
$\lambda_k(G)=1$ ? Actually, one can check that a connected graph
$G$ obtained from the complete graph $K_{n-1}$ by attaching a
pendant edge is a such graph, which is obviously a unique maximum
such graph. However, to characterize all the maximal graphs is left
unsolved. Here maximal (minimal) means that adding (deleting) any
edge with destroy $\kappa_k(G)=1$ or $\lambda_k(G)=1$, whereas
maximum means a such graph that has the largest number of edges.

\subsection{Nordhaus-Gaddum-type results}

Alavi and Mitchem in \cite{Alavi} considered the
Nordhaus-Gaddum-type results for the connectivity and
edge-connectivity. We are concerned with analogous inequalities
involving the generalized $k$-connectivity.

\begin{theorem}\label{th6}
For any graph $G$ of order $n$, we have

$(1)$ $1\leq \kappa_k(G)+\kappa_k(\overline{G})\leq n-\lceil k/2
\rceil$;

$(2)$ $0\leq \kappa_k(G)\cdot \kappa_k(\overline{G})\leq
[\frac{n-\lceil k/2 \rceil}{2}]^2$.

Moreover, the upper and lower bounds are tight.
\end{theorem}

\begin{proof}
$(1)$ Since $G\cup \overline{G}=K_n$,
$\kappa_k(G)+\kappa_k(\overline{G})\leq \kappa_k(K_n)$. This
together with $\kappa_k(K_n)=n-\lceil\frac{k}{2}\rceil$ results in
$\kappa_k(G)+\kappa_k(\overline{G})\leq n-\lceil\frac{k}{2}\rceil$.
If $\kappa_k(G)+\kappa_k(\overline{G})=0$, then
$\kappa_k(G)=\kappa_k(\overline{G})=0$. Thus $G$ and $\overline{G}$
are all disconnected, which is impossible. So
$\kappa_k(G)+\kappa_k(\overline{G})\geq 1$.

$(2)$ It follows immediately from $(1)$.
\end{proof}

To see that the lower bound of $(1)$ is tight, it suffices to take
$G$ as the complete bipartite graph $K_{1,n-1}$ since
$\kappa_k(K_{1,n-1})+\kappa_k(\overline{K_{1,n-1}})=1+0=1$.

The following observation indicates the graphs attaining the lower
bound of $(2)$.

\begin{observation}\label{obs3}
$\kappa_k(G)\cdot \kappa_k(\overline{G})=0$ if and only if $G$ or
$\overline{G}$ is disconnected.
\end{observation}

We construct a graph class to show that the two upper bounds of are
tight for $k=n$.

\noindent \textbf{Example 3.}~~Let $n,r$ be two positive integers
such that $n=4r+1$. From Lemma \ref{lem3}, we know that
$\kappa_{n}(K_{2r,2r+1})=\lambda_{n}(K_{2r,2r+1})=r$. Let $\mathcal
{E}$ be the set of the edges of these $r$ spanning trees in
$K_{2r,2r+1}$. Then there exist $2r(2r+1)-4r^2=2r$ remaining edges
in $K_{2r,2r+1}$ except the edges in $\mathcal {E}$. Let $M$ be the
set of these $2r$ edges. Set $G=K_{2r,2r+1}\setminus M$. Then
$\kappa_{n}(G)=r$, $M\subseteq E(\overline{G})$ and $\overline{G}$
is a graph obtained from two cliques $K_{2r}$ and $K_{2r+1}$ by
adding $2r$ edges in $M$ between them, that is, one end of each edge
belongs to $K_{2r}$ and the other belongs to $K_{2r+1}$. Note that
$E(\overline{G})=E(K_{2r})\cup M\cup E(K_{2r+1})$. Now we show that
$\kappa_{n}(\overline{G})\geq r$. As we know, $K_{2r}$ contains $r$
Hamiltonian paths, say $P_{1},P_{2},\cdots,P_{r}$, and so does
$K_{2r+1}$, say $P_{1}',P_{2}',\cdots,P_{r}'$. Pick up $r$ edges
from $M$, say $e_1,e_2,\cdots,e_r$, let $T_i=P_{i}\cup P_{i}'\cup
e_i(1\leq i\leq r)$. Then $T_{1},T_{2},\cdots,T_{r}$ are $r$
spanning trees in $\overline{G}$, namely,
$\kappa_{n}(\overline{G})\geq r$. Since
$|E(\overline{G})|={{2r}\choose{2}}+{{2r+1}\choose{2}}+2r=4r^2+2r$
and each spanning tree uses $4r$ edges, these edges can form at most
$\lfloor\frac{4r^2+2r}{4r}\rfloor=r$ spanning trees, that is,
$\kappa_{n}(\overline{G})\leq r$. So $\kappa_{n}(\overline{G})=r$.
Clearly, $\kappa_{n}(G)+\kappa_{n}(\overline{G})=2r=\frac{n-1}{2}
=n-\lceil\frac{n}{2}\rceil$ and $\kappa_{n}(G)\cdot
\kappa_{n}(\overline{G})=r^2=\big[\frac{n-\lceil n/2
\rceil}{2}\big]^2$.

\noindent \textbf{Remark 2}. The above example only shows that the
upper bound of (2) in Theorem \ref{th6} is tight for the case $k=n$.
A natural question is to find examples showing that the upper bounds
of Theorem \ref{th6} are tight for each $k$ with $3\leq k<n$. Note
that the complete graph $G=K_n$ can attain the upper bound of $(1)$,
but clearly $\overline{G}$ is disconnected. Therefore, when we
require that both $G$ and $\overline{G}$ are connected, is there a
graph which can attain the upper bounds of Theorem \ref{th6}
respectively or simultaneously for each $k$ with $3\leq k\leq n$ ?

\section{Results for $\lambda_3(G)$ and $\kappa_3(G)$}

\subsection{Upper and lower bounds for $\lambda_3(G)$}

From now on, we focus our attention on the generalized
$3$-edge-connectivity. From Proposition \ref{pro2}, we obtained
tight upper and lower bounds of $\lambda_3(G)$, that is, $1\leq
\lambda_3(G)\leq n-2$. Now we give another tight upper and lower
bounds of $\lambda_3(G)$ by the edge-connectivity, that is,
$\frac{3\lambda-2}{4}\leq \lambda_3(G)\leq \lambda$, which will be
used in planar graph and line graph. At first we give a tight upper
bound for $\lambda_k(G)$.

\begin{proposition}\label{pro3}
For any graph $G$ of order $n$, $\lambda_k(G)\leq \lambda(G)$.
Moreover, the upper bound is tight.
\end{proposition}

\begin{proof}
Let $M$ be a $\lambda(G)$-edge-cut of $G$, where $1\leq
\lambda(G)\leq n-1$. Then $G\setminus M$ has at least two
components. We can choose $S=\{v_1,v_2,\cdots,v_k\}$ so that
$S\subseteq V(G)$ and at least two of the $k$ vertices are in
different components. Thus any tree connecting $S$ must contain an
edge in $M$. By the definition of $\lambda(S)$, we get
$\lambda(S)\leq |M|$. So $\lambda_k(G)\leq \lambda(S)\leq
|M|=\lambda(G)$.

Furthermore, we will show that the graph $G=K_k\vee(n-k)K_1 \ (n\geq
3k)$ satisfies that
$\kappa_k(G)=\lambda_k(G)=\kappa(G)=\lambda(G)=\delta(G)=k$ (See
Figure $1$).

\begin{figure}[h,t,b,p]
\begin{center}
\scalebox{0.7}[0.7]{\includegraphics{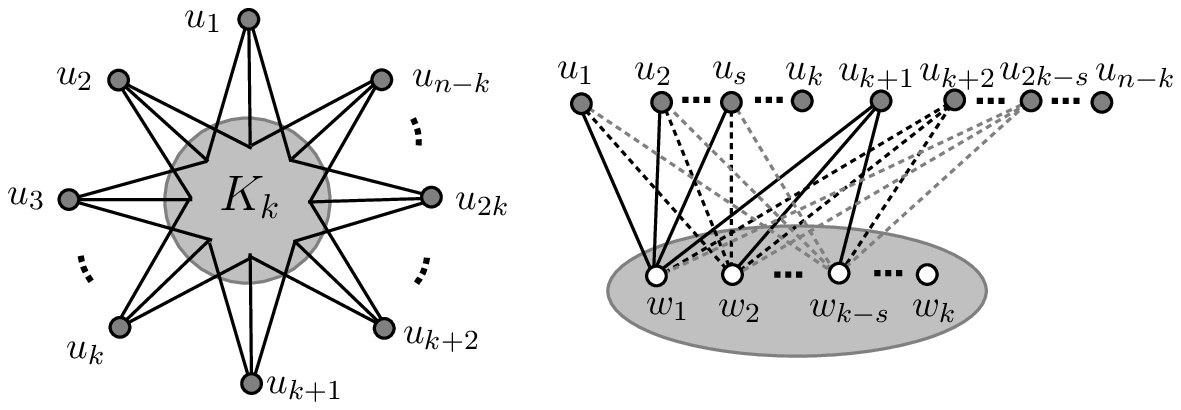}}\\
Figure 1: Graph $G$ with
$\kappa_k(G)=\lambda_k(G)=\kappa(G)=\lambda(G)=\delta(G)=k$.
\end{center}
\end{figure}

Let $W=\{w_1,w_2,\cdots,w_{k}\}$, $U=K_k\setminus
W=\{u_1,u_2,\cdots,u_{n-k}\}$, and $S$ be a $k$-subset of vertices
of $G$. Without loss of generality, let $|S\cap V(U)|=s \ (s\leq
k)$. Then $|S\cap V(W)|=k-s$. Without loss of generality, let
$u_i\in S \ (1\leq i\leq s)$ and $w_j\in S \ (1\leq j\leq k-s)$.
Then the trees $T_i=w_iu_1\cup w_iu_2\cup \cdots \cup w_iu_s \cup
u_{k+i}w_1\cup u_{k+i}w_2\cup \cdots \cup
u_{k+i}w_{k-s}(i=1,2,\cdots,k-s)$ and $T_j=w_ju_1\cup w_ju_2\cup
\cdots \cup w_ju_s \cup w_jw_1\cup w_jw_2\cup \cdots \cup w_jw_{k-s}
\ (j=k-s+1,k-s+2,\cdots,k)$ form $k$ pairwise edge-disjoint trees
connecting $S$, namely $\lambda(S)\geq k$. Combining this with
$\lambda_k(G)\leq \lambda(G)=k$, we get $\lambda_k(G)=k$. Since the
above $k$ trees are also internally disjoint trees connecting $S$,
we have $\kappa_k(G)=k$. So
$\kappa_k(G)=\lambda_k(G)=\kappa(G)=\lambda(G)=\delta(G)=k$.
Clearly, the upper bound of Proposition \ref{pro3} is tight.
\end{proof}

Next we give a tight lower bound for $\lambda_3(G)$.
\begin{proposition}\label{pro4}
Let $G$ be a connected graph with $n$ vertices. For every two
integers $s$ and $r$ with $s\geq 0$ and $r\in \{0,1,2,3\}$, if
$\lambda(G)=4s+r$, then $\lambda_3(G)\geq
3s+\lceil\frac{r}{2}\rceil$. Moreover, the lower bound is tight. We
simply write $\lambda_3(G)\geq \frac{3\lambda-2}{4}$.
\end{proposition}
\begin{proof}
Let $\lambda=\lfloor\frac{8t+3}{6}\rfloor$. From Lemma \ref{lem1},
we have $\lambda_3(G)\geq t$ (Note that we can use the result of
Lemma \ref{lem1}, although Kriesell \cite{Kriesell1} considered
graphs containing multiple edges but here we are concerned with the
generalized edge-connectivity for simple graphs).

If $\lambda=4s$, since $\frac{8t+3}{6}$ is not an integer, then
$4s<\frac{8t+3}{6}$. Thus $\lambda_3(G)\geq t>3s-\frac{3}{8}$, which
implies $\lambda_3(G)\geq 3s$. With a similar method, we can obtain
that $\lambda_3(G)\geq 3s+1$ if $\lambda=4s+1$, and $\lambda_3(G)
\geq 3s+2$ if $\lambda=4s+3$.

Note that there exists no integer $t$ such that
$4s+2=\lfloor\frac{8t+3}{6}\rfloor$ if $\lambda=4s+2$. But a graph
$G$ with $\lambda(G)=4s+2$ is also $(4s+1)$-edge-connected, and so
we have $\lambda_3(G)\geq 3s+1$.
$$
\lambda_3(G)\geq \left\{
\begin{array}{ll}
3s&if~\lambda=4s,\\
3s+2&if~\lambda=4s+3,\\
3s+1&if~\lambda=4s+1~or~\lambda=4s+2.
\end{array}
\right.
$$
So the result holds. Simply, we write $\lambda_3(G)\geq
\frac{3\lambda-2}{4}$.

Now we give graphs attaining the lower bound.

For $\lambda=4s$ with $s\geq 1$, we construct a graph $G$ as follows
(see Figure 2 $(a)$): Let $P=X_1\cup X_2$ and $Q=Y_1\cup Y_2$ be two
cliques with $|X_1|=|Y_1|=2s$ and $|X_2|=|Y_2|=2s$. Let $u,v$ be
adjacent to every vertex in $P, Q$, respectively, and $w$ be
adjacent to every vertex in $X_1$ and $Y_1$. Finally, we finish the
construction of the graph $G$ by adding a perfect matching between
$X_2$ and $Y_2$. It can be easily checked that $\lambda=4s$.

\begin{figure}[h,t,b,p]
\begin{center}
\scalebox{0.8}[0.8]{\includegraphics{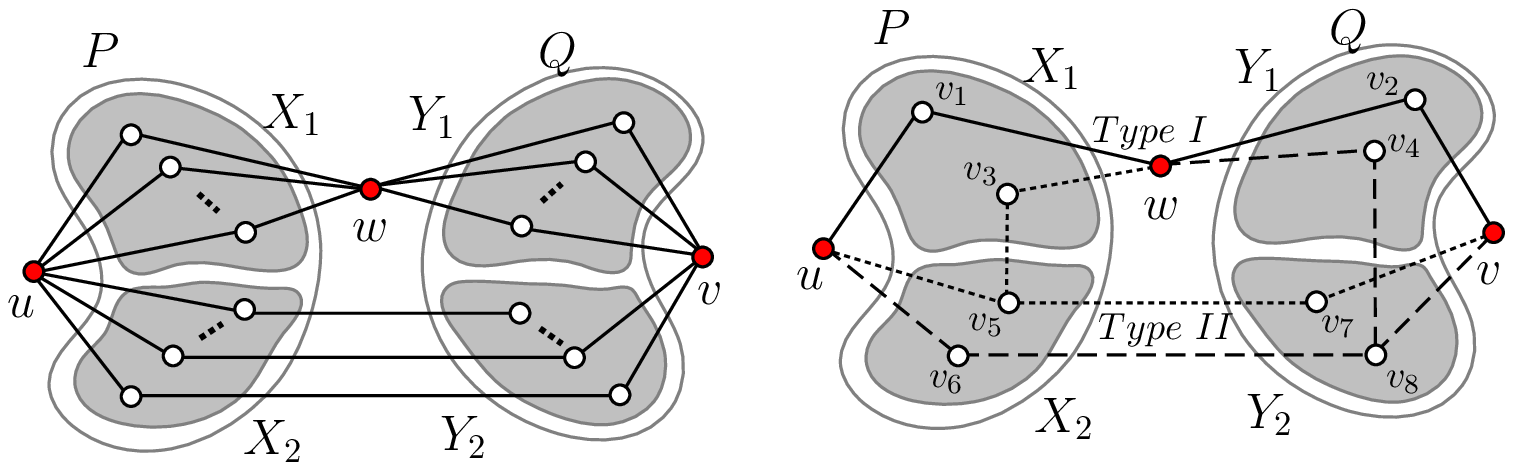}}\\
Figure 2 (a): The graph with $\lambda(G)=4s$ and $\lambda_3(G)=3s$.\\
Figure 2 (b): Two types of trees connecting $\{u,v,w\}.$
\end{center}
\end{figure}

We consider the case $S=\{u,v,w\}$. There exist two kinds of
edge-disjoint trees connecting $S$ (see Figure 2 $(b)$): the tree of
Type $I$ is a path $u$-$v_1$-$w$-$v_2$-$v$; the tree of Type $II$ is
$T_1$ or $T_2$, where $T_1=uv_5\cup v_3v_5\cup wv_3\cup v_5v_7\cup
v_7v$ and $T_2=uv_6\cup v_6v_8\cup v_8v_4\cup v_4w\cup v_8v$,
respectively. We denote the numbers of trees of Type $I$ and Type
$II$ by $x$ and $y$, respectively. Note that $|E_G[w,X_1\cup
Y_1]|=4s$ and each tree of Type $I$ uses two edges of $E_G[w,X_1\cup
Y_1]$, we have $x\leq 2s$. Although each tree of Type $II$ uses one
edge of $E_G[w,X_1\cup Y_1]$, we have $y\leq 2s$ since each tree of
Type $II$ uses one edge of $E_G[X_2,Y_2]$ and $|E_G[X_2,Y_2]|=2s$.
Combining these with $2x+y\leq 4s$, we can derive the optimal
solution $x=s$ and $y=2s$ by solving the following integer linear
programming:
$$\left\{\begin{array}{ll}
 Maximize:~x+y\\Subject~to:~x\leq 2s,~y\leq 2s,~2x+y\leq 4s,\\
 and~~~~~~~~~~~~x,y\geq 0.
 \end{array}\right.
 $$
Thus $\lambda(S)\geq 3s$. We can check that for any other three
vertices of $G$ the number of edge-disjoint trees connecting them is
not less than $3s$. So $\lambda_3(G)=3s$ and the graph $G$ attaining
the lower bound.

For $\lambda=4s+1$, let $|X_1|=|Y_1|=2s+1$ and $|X_2|=|Y_2|=2s$; for
$\lambda=4s+2$, let $|X_1|=|Y_1|=2s+1$ and $|X_2|=|Y_2|=2s+1$; for
$\lambda=4s+3$, let $|X_1|=|Y_1|=2s+2$ and $|X_2|=|Y_2|=2s+1$, where
$s\geq 1$. Similarly, we can check that $\lambda_3(G)=3s+1$ for
$\lambda=4s+1$; $\lambda_3(G)=3s+1$ for $\lambda=4s+2$;
$\lambda_3(G)=3s+2$ for $\lambda=4s+3$.

For the case $s=0$, we have $G=P_n$ such that
$\lambda(G)=\lambda_3(G)=1$; $G=C_n$ such that $\lambda(G)=2$ and
$\lambda_3(G)=1$; $G=H_t$ such that $\lambda(G)=3$ and
$\lambda_3(G)=2$, where $H_t$ denotes the graph obtained from $t$
copies of $K_4$ by identifying a vertex from each of them in the way
shown in Figure $3$.

\begin{figure}[h,t,b,p]
\begin{center}
\scalebox{0.8}[0.8]{\includegraphics{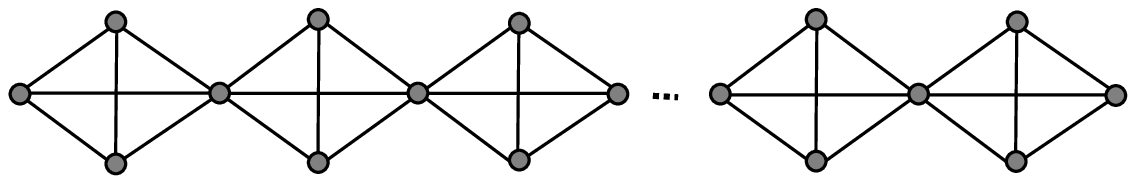}}\\
Figure 3: $\lambda(H_t)=3,\lambda_3(H_t)=2$.
\end{center}
\end{figure}
\end{proof}

As we know, every planar graph $G$ has a vertex of degree at most
$5$, i.e., $\delta(G)\leq 5$. Since $\lambda(G)\leq \delta$, we only
need to consider a planar graph $G$ with edge-connectivity
$\lambda(G)$ at most $5$. From Proposition \ref{pro4}, it can be
deduced that for any graph (not necessarily planar) if
$\lambda(G)=1$, $\lambda_3(G)=1$; if $\lambda(G)=2$,
$\lambda_3(G)\geq 1$; if $\lambda(G)=3$, $\lambda_3(G)\geq 2$; if
$\lambda(G)=4$, $\lambda_3(G)\geq 3$; and if $\lambda(G)=5$,
$\lambda_3(G)\geq 4$. Therefore, the following corollary is obvious.
\begin{corollary}\label{cor2}
If $G$ is a connected planar graph, then $\lambda(G)-1\leq
\lambda_3(G)\leq \lambda(G)$.
\end{corollary}

\subsection{Results for line graphs}

This section investigate the relation between the generalized
$3$-connectivity and generalized $3$-edge-connectivity of a graph
and its line graph.

\begin{proposition}\label{pro5}
If $G$ is a connected graph, then

$(1)$ $\lambda_3(G)\leq \kappa_3(L(G))$.

$(2)$ $\lambda_3(L(G))\geq \frac{3}{2}\lambda_3(G)-2$.

$(3)$ $\kappa_3(L(L(G))\geq \frac{3}{2}\kappa_3(G)-2$.
\end{proposition}
\begin{proof}
For $(1)$, let $e_1,e_2,e_3$ be three arbitrary distinct vertices of
the line graph of $G$ such that $\lambda_3(G)=t$ with $t\geq 1$. Let
$e_1=v_1v_1'$, $e_2=v_2v_2'$ and $e_3=v_3v_3'$ be those edges of $G$
corresponding to the vertices $e_1,e_2,e_3$ in $L(G)$, respectively.

Consider three distinct vertices of the six end-vertices of
$e_1,e_2,e_3$. Without loss of generality, let $S=\{v_1,v_2,v_3\}$
be three distinct vertices. Since $\lambda_3(G)=t$, there exist $t$
edge-disjoint trees $T_1,T_2,\cdots,T_t$ connecting $S$ in $G$. We
define a minimal tree $T$ connecting $S$ as a tree connecting $S$
whose subtree obtained by deleting any edge of $T$ does not connect
$S$.

\begin{figure}[h,t,b,p]
\begin{center}
\scalebox{0.7}[0.7]{\includegraphics{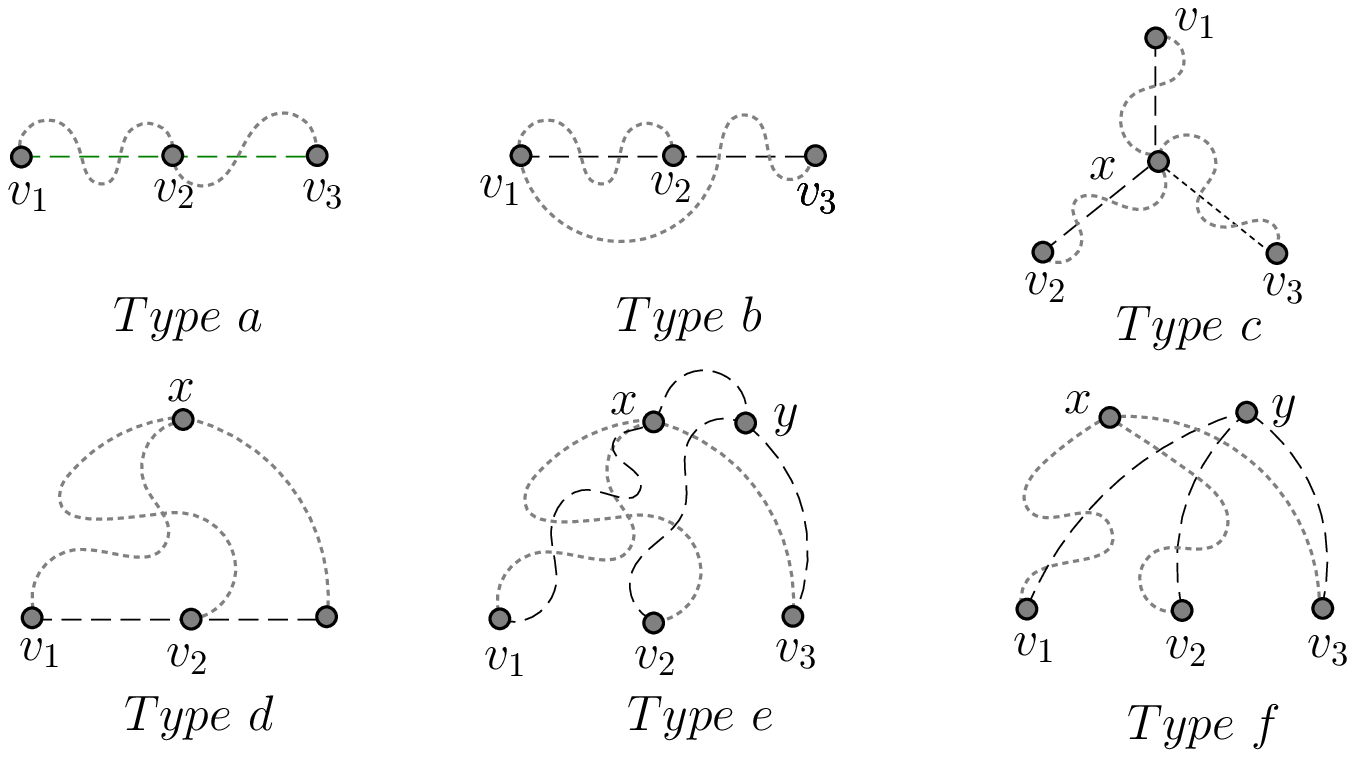}}\\
Figure 4:  Six possible types of $T_i\cup T_j$.
\end{center}
\end{figure}

Choosing any two edge-disjoint minimal trees $T_i$ and $T_j$ $(1\leq
i,j\leq t)$ connecting $S$ in $G$, we will show that the trees
$T_i'$ and $T_j'$ corresponding to $T_i$ and $T_j$ in $L(G)$ are
internally disjoint trees. It is easy to see that $T_i\cup T_j$ has
six possible types, as shown in Figure $4$. Since $T_i$ and $T_j$
are edge-disjoint in $G$, we can find internally disjoint trees
$T_i'$ and $T_j'$ connecting $e_1,e_2,e_3$ in $L(G)$. We give an
example of Type $c$, see Figure $5$. So $\kappa_3(L(G))\geq t$ and
we know that the result holds.

\begin{figure}[h,t,b,p]
\begin{center}
\scalebox{0.7}[0.7]{\includegraphics{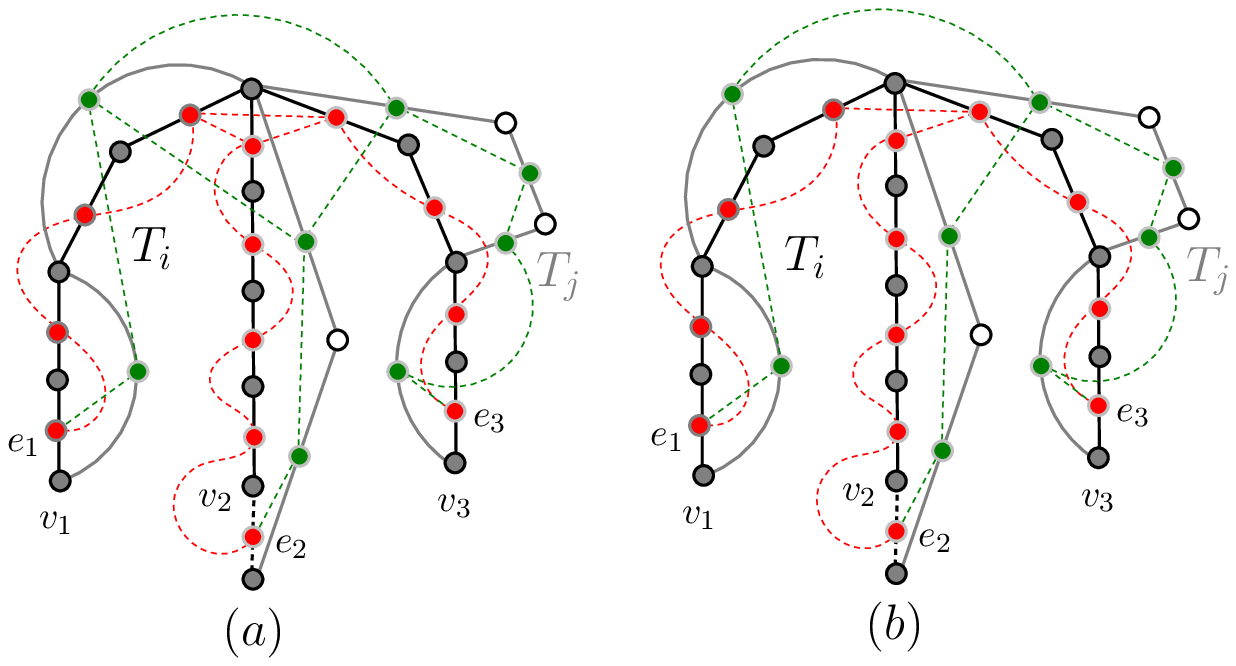}}\\
Figure 5 (a): An example for $T_i$ and $T_j$ connecting $S$ and their line graphs. \\
Figure 5 (b): An example for $T_i'$ and $T_j'$ corresponding to
$T_i$ and $T_j$.
\end{center}
\end{figure}

For $(2)$, from Propositions \ref{pro3} and \ref{pro4} and $(2)$ of
Lemma \ref{lem2} we have that $\lambda_3(L(G))\geq
\frac{3}{4}\lambda(L(G))-\frac{1}{2} \geq
\frac{3}{4}(2\lambda(G)-2)-\frac{1}{2}=\frac{3}{2}\lambda(G)-2\geq
\frac{3}{2}\lambda_3(G)-2$.

For $(3)$, from $(1)$ and $(2)$ of this proposition and Observation
\ref{obs1} we have that $\kappa_3(L(L(G)))\geq \lambda_3(L(G))\geq
\frac{3}{2}\lambda_3(G)-2 \geq\frac{3}{2}\kappa_3(G)-2$.

One can check that $(1)$ of this proposition is tight since $G=C_n$
can attain this bound.
\end{proof}

Let $L^0(G)=G$ and $L^1(G)=L(G)$. Then for $k\geq 2$, the $k$-$th$
iterated line graph $L^k(G)$ is defined by $L(L^{k-1}(G))$. The next
statement follows immediately from Proposition \ref{pro5} and a
routine application of recursions.

\begin{corollary}
$\lambda_3(L^k(G))\geq (\frac{3}{2})^k(\kappa_3(G)-4)+4$, and
$\kappa_3(L^k(G))\geq
(\frac{3}{2})^{\lfloor\frac{k}{2}\rfloor}(\kappa_3(G)-4)+4$.
\end{corollary}

\end{document}